\documentclass[11pt,a4paper]{article}
\usepackage[utf8]{inputenc}
\usepackage[english]{babel}
\usepackage{amsmath}
\usepackage{amsfonts}
\usepackage{amssymb}
\usepackage{amsthm}
\usepackage[left=2cm,right=2cm,top=2cm,bottom=2cm]{geometry}

\usepackage[pdftex]{graphicx}
\usepackage{tikz}
\usepackage{color}
\usepackage{subfigure}

\usepackage{hyperref}

\usepackage{mathtools} 
\mathtoolsset{showonlyrefs,showmanualtags} 

\usepackage[auth-sc]{authblk}

\newcommand{\tx}[1]{\mathrm{#1}}

\DeclareMathOperator{\rank}{\tx{rank}}
\DeclareMathOperator{\spn}{\tx{span}}

\DeclareMathOperator{\ind}{\tx{ind}}

\renewcommand{\H}{\mathbf{H}}
\renewcommand{\L}{\mathcal{L}}
\newcommand{\J}{\Omega}
\newcommand{\R}{\mathbb{R}}
\newcommand{\N}{\mathbb{N}}
\newcommand{\eps}{\varepsilon}
\newcommand{\ve}{\mathcal{V}}

\newcommand{\HH}{\mathcal{H}}

\newtheorem{mainthm}{Theorem} 

\newtheorem{thm}{Theorem}[section]
\newtheorem{lem}[thm]{Lemma}
\newtheorem{cor}[thm]{Corollary}
\newtheorem{prop}[thm]{Proposition}
\theoremstyle{definition}
\newtheorem{definition}[thm]{Definition}
\theoremstyle{remark}
\newtheorem{rem}{Remark}

\title{On conjugate times of LQ optimal control problems}

\date{\today}

\author[1,2]{Andrei Agrachev}
\author[3]{Luca Rizzi}
\author[1]{Pavel Silveira}
\affil[1]{\small SISSA, Trieste, Italy}
\affil[2]{\small Steklov Mathematical Institute, Moscow, Russia}
\affil[3]{\small CNRS, CMAP \'Ecole Polytechnique and \'Equipe INRIA GECO Saclay \^Ile-de-France, Paris, France}

\begin{document}

\maketitle

\begin{abstract}
Motivated by the study of linear quadratic optimal control problems, we consider a dynamical system with a constant, quadratic Hamiltonian, and we characterize the number of conjugate times in terms of the spectrum of the Hamiltonian vector field $\vec{H}$. We prove the following dichotomy: the number of conjugate times is identically zero or grows to infinity. The latter case occurs if and only if $\vec{H}$ has at least one Jordan block of odd dimension corresponding to a purely imaginary eigenvalue. As a byproduct, we obtain bounds from below on the number of conjugate times contained in an interval in terms of the spectrum of $\vec{H}$.
\end{abstract}

\tableofcontents

\section{Introduction}\label{s:intro}

Linear Quadratic optimal control problems (LQ in the following) are a standard topic in control theory and dynamical systems, and are very popular in applications. They consist in a linear control system with quadratic Lagrangian. We briefly recall the general features of a LQ problem, and we refer to \cite[Chapter 16]{AAAbook} and \cite[Chapter 7]{MR1425878} for further details. We are interested in \emph{admissible trajectories}, namely  curves $x:[0,t_1]\to \mathbb{R}^n$ such that there exists a control $u \in L^2([0,t_1],\mathbb{R}^k)$ such that
\begin{equation}
\dot{x} = Ax +Bu, \qquad x(0) = x_0, \qquad x(t_1) = x_1,\qquad x_0,x_1,t_1 \text{ fixed},
\end{equation}
that minimize a quadratic functional $\phi_{t_1}: L^2([0,t_1],\mathbb{R}^k) \to \mathbb{R}$ of the form
\begin{equation}
\phi_{t_1}(u) = \frac{1}{2}\int_0^{t_1} \left(u^*Ru + x^*P u + x^* Q x \right) dt.
\end{equation}
The condition $R\geq 0$ is necessary for existence of optimal control. We also assume $R>0$ (for the singular case we refer to \cite[Chapter 9]{MR1425878}). Without loss of generality we may reduce to the case

\begin{equation}
\phi_{t_1}(u) = \frac{1}{2}\int_{0}^{t_1} \left(u^* u - x^*Qx \right)dt.
\end{equation}
Here $A,B,Q$ are constant matrices of the appropriate dimension. The vector $Ax$ represents the \emph{drift} field, while the columns of $B$ represent the controllable directions. The meaning of the \emph{potential} term $Q$ will be clear later, when we will introduce the Hamiltonian associated with the LQ problem.

We assume that the system is \emph{controllable}, namely there exists $m >0$ such that
\begin{equation}
\rank(B,AB,\ldots,A^{m-1}B) = n.
\end{equation}
This hypothesis implies that, for any choice of $t_1,x_0,x_1$, the set of controls $u$ such that the associated trajectory $x_u :[0,t_1] \to \mathbb{R}^n$ connects $x_0$ with $x_1$ in time $t_1$ is non-empty.

It is well known that the optimal trajectories of the LQ system are projections $(p,x) \mapsto x$ of the solutions of the Hamiltonian system
\begin{equation}
\dot{p}  = -\partial_x H (p,x), \qquad \dot{x} = \partial_p H (p,x), \qquad (p,x) \in T^*\R^n = \R^{2n},
\end{equation}
where the Hamiltonian function $H: \mathbb{R}^{2n} \to \mathbb{R}$ is defined by
\begin{equation}\label{eq:Hamiltonian}
H(p,x) = \frac{1}{2}(p,x)^* \H \begin{pmatrix}
p \\ x
\end{pmatrix}, \qquad \H = \begin{pmatrix}
BB^* & A \\ A^* & Q
\end{pmatrix}.
\end{equation}
We denote by $P_t : \R^{2n} \to \R^{2n}$ the flow of the Hamiltonian system, which is defined for all $t \in \R$. 
To exploit the natural symplectic setting on $T^*\R^n = \R^{2n}$, we employ canonical coordinates $(p,x)$ such that the symplectic form $\omega = \sum_{i=1}^n dp_i \wedge dx_i$ is represented by the matrix $\J = \left(\begin{smallmatrix}
0 & \mathbb{I}_n \\ -\mathbb{I}_n & 0
\end{smallmatrix}\right)$. The flow lines of $P_t$ are precisely the integral lines of the \emph{Hamiltonian vector field} $\vec{H} \in \text{Vec}(\mathbb{R}^{2n})$, defined by $dH(\cdot) = \omega(\,\cdot\,,\vec{H})$. More explicitly
\begin{equation}
\vec{H}_{(p,x)}=\begin{pmatrix} -A^* & -Q \\ BB^* & A \end{pmatrix}\begin{pmatrix}
p \\ x
\end{pmatrix} = -\J\H \begin{pmatrix}
p \\ x
\end{pmatrix}.
\end{equation}
By the term Hamiltonian vector field, we denote both the linear field $\vec{H}$ and the associated matrix $-\J\H$. The Hamiltonian flow can be explicitly written in terms of the latter as
\begin{equation}
P_t = e^{-t\J\H},
\end{equation}
where the r.h.s. is the standard matrix exponential.

\subsection*{Conjugate times}

We stress that not all the integral lines of the Hamiltonian flow lead to minimizing solutions of the LQ problem, since they only satisfy first order conditions for optimality. For this reason, they are usually called \emph{extremals}. Sufficiently short segments, however, are optimal, but they lose optimality at some time $t>0$, called the \emph{first conjugate time}. In the following, we give a geometrical definition of conjugate time, in terms of curves in the Grassmannian of Lagrangian subspaces of $\R^{2n}$.

We say that a subspace $\Lambda \subset \R^{2n}$ is \emph{Lagrangian} if $\omega|_\Lambda \equiv 0$, and $\dim \Lambda = n$. A notable example of Lagrangian subspace is the \emph{vertical} subspace, that is $\ve := \{(p,0)|\,p \in \R^n\}$.

\begin{definition}
The Jacobi curve $J(\cdot)$ is the following family of Lagrangian subspaces of $\R^{2n}$
\begin{equation}
J(t):= e^{t\J\H}\ve, \qquad \ve :=\{(p,0)|\,p \in \R^n\}.
\end{equation}
From the geometrical viewpoint, $J(\cdot)$ is a smooth curve in the submanifold of the Grassmannian of the $n$-dimensional subspaces of $\mathbb{R}^{2n}$ defined by the Lagrangian subspaces. 
\end{definition}
\begin{definition}
We say that $t$ is a conjugate time if $J (t) \cap \ve \neq 0$. The \emph{multiplicity} of the conjugate time $t$ is the dimension of the intersection.
\end{definition}
In the language introduced by V. Arnold, these are times of \emph{verticality} of the Jacobi curve. It is not hard to show that $t$ is a conjugate time if and only if there exist solutions of the Hamilton equations such that $x(0) = x(t) = 0$.

We briefly recall the connection between conjugate times and second order conditions for optimality. The solutions of the LQ problems can be seen as constrained minima of the quadratic functional $\phi_{t_1}$ on $\mathcal{U}(x_0,x_1) \subset L^2([0,t_1],\R^k)$ given by all the controls $u$ such that $x_u(0) = x_0$ and $x_u(t_1) = x_1$. It is easy to check that $\mathcal{U}(x_0,x_1) = u^* + \mathcal{U}(0,0)$ for any $u^* \in \mathcal{U}(x_0,x_1)$, that is $\mathcal{U}(x_0,x_1)$ is an affine space over the vector space $\mathcal{U}(0,0)$. For this reason, the behaviour of $\phi_{t_1}$, restricted to $\mathcal{U}(0,0)$ provides all the informations about optimality. It is a well known fact that the number of conjugate times in the interval $(0,t_1)$, counted with their multiplicity, is equal to the negative inertia index of the quadratic form $\phi_{t_1} : \mathcal{U}(0,0) \to \R$ (this can be proved directly with the techniques in~\cite[Propositions 16.2, 16.3]{AAAbook}, see also~\cite[Theorem I.2]{NoteCime} for a more general setting). The occurrence of conjugate times implies that an extremal cannot be a minimizer, since one can find a small variation of $u^*$ that decreases the value of $\phi_{t_1}$. The first conjugate time determines existence and uniqueness of minimizing solutions of the LQ problem, as specified by the following proposition.

\begin{prop}
Let $\bar{t}$ be the first conjugate time, namely $\bar{t} := \inf\{t>0|\,J(t) \cap \ve \neq 0\}$. 
\begin{itemize}
\item For $t_1<\bar{t}$, for any $x_0,x_1$ there exists a unique minimizer connecting $x_0$ with $x_1$ in time $t_1$. 
\item For $t_1>\bar{t}$, for any $x_0,x_1$ there exists no minimizer connecting $x_0$ with $x_1$ in time $t_1$.
\item For $t_1= \bar{t}$, existence of minimizers depends on the initial data.
\end{itemize}
\end{prop}

In this paper we completely characterise the occurrence of conjugate times for a controllable LQ problem. In particular, we prove the following result.

\begin{mainthm}\label{t:main}
The conjugate times of a controllable linear quadratic optimal control problem obey the following dichotomy:
\begin{itemize}
\item If the Hamiltonian field $\vec{H}$ has at least one odd-dimensional Jordan block corresponding to a pure imaginary eigenvalue, the number of conjugate times in the interval $[0,T]$ grows to infinity for $T\to \pm\infty$.
\item If the Hamiltonian field $\vec{H}$ has no odd-dimensional Jordan blocks corresponding to a pure imaginary eigenvalue, there are no conjugate times.
\end{itemize}
\end{mainthm}
In Sec.~\ref{s:main}, we also provide estimates for the first conjugate time, in terms of the (signed) eigenvalues of $\vec{H}$ (see Corollaries~\ref{c:estimate1} and~\ref{c:estimate2}).

Before passing to a more detailed description of curves of Lagrangian subspaces, we stress that the concept of Jacobi curves is not limited to LQ optimal control problems and can be defined for way more general geometrical structures, such as control systems with Tonelli Lagrangian including, among the others Riemannian, sub-Riemannian, Finsler and sub-Finsler manifolds. In these more general settings, however, we cannot exploit the natural linear structure of $\R^n$, and the Jacobi curve is a curve of subspaces of the tangent space to the cotangent bundle, associated with a fixed ``geodesic'' (i.e. locally minimizing curve) of the underlying structure. We refer the interested reader to \cite{curvature,geometryjacobi1,agrafeedback}.

The plan of the paper is as follows. In Sec.~\ref{s:preliminaries} we recall some basic facts about geometry of curves in Lagrange Grassmannian, and the main technical tool: the Maslov index. Then, in Sec.~\ref{s:main} we prove the main result.

\section{Curves in the Lagrange Grassmannian}\label{s:preliminaries}

Let $(\Sigma, \omega)$ be a $2n$-dimensional symplectic vector space. Recall that subspace $\Lambda  \subset \Sigma$ is called 
\emph{Lagrangian} if it has dimension $n$ and $\omega|_{\Lambda}\equiv 0.$  The \emph{Lagrange Grassmannian} $\L(\Sigma)$ is the set of all $n$-dimensional Lagrangian subspaces of $\Sigma$. 

\begin{prop}\label{p:lagrass} 
$\L(\Sigma)$ is a compact $n(n+1)/2$-dimensional submanifold of the Grassmannian of $n$-planes in $\Sigma$.
\end{prop}
\begin{proof}
Let $\Delta \in \L(\Sigma)$, and consider the set $\Delta^{\pitchfork}:= \{\Lambda \in \L(\Sigma)\,|\,  \Lambda \cap \Delta =0 \}$ of all Lagrangian subspaces transversal to $\Delta$.
Clearly, the collection of these sets for all $\Delta \in \L(\Sigma)$ is an open cover of $\L(\Sigma)$. Then it is  sufficient to find submanifold coordinates on each $\Delta^{\pitchfork}$. 

Let us fix any Lagrangian complement $\Pi$ of $\Delta$ (which always exists, though it is not unique).
Every $n$-dimensional subspace $\Lambda \subset \Sigma$ that is transversal to $\Delta$ is the graph of a linear map from $\Pi$ to $\Delta$. Choose an adapted Darboux basis on $\Sigma$, namely a basis $\{e_i,f_i\}_{i=1}^n$ such that
\begin{gather}
\Delta = \spn\{f_1,\ldots,f_n\}, \qquad \Pi = \spn\{e_1,\ldots,e_n\}, \\
\omega(e_i,f_j) - \delta_{	ij} = \omega(f_i,f_j) = \omega(e_i,e_j) = 0, \qquad i,j=1,\ldots,n.
\end{gather} 
In these coordinates, the linear map is represented by a matrix $S_{\Lambda}$ such that
\begin{equation} 
\Lambda \cap \Delta=0 \Leftrightarrow \Lambda=\{(p,S_{\Lambda} p)| \,p \in \Pi\simeq \mathbb{R}^{n}\}.  
\end{equation}
Moreover it is easy to see that $\Lambda \in \L(\Sigma) $ if and only if $ S_{\Lambda}=S_{\Lambda}^{*}$.
Hence, the open set $\Delta^{\pitchfork}$ of all Lagrangian subspaces transversal to $\Delta$ is parametrized by the set of symmetric matrices, and this gives smooth submanifold coordinates on $\Delta^\pitchfork$.
This also proves that the dimension of $\L(\Sigma)$ is $n(n+1)/2$. Finally, as a closed subset of a compact manifold, $\L(\Sigma)$ is compact.
\end{proof}

Fix now $\Lambda \in \L(\Sigma)$. The tangent space $T_{\Lambda }\L(\Sigma)$ to the Lagrange Grassmannian at the point $\Lambda$ can be canonically identified with the set of quadratic forms on the space $\Lambda$ itself, namely
\begin{equation}
T_{\Lambda}\L(\Sigma)\simeq Q(\Lambda).
\end{equation}
Indeed, consider a smooth curve $\Lambda(\cdot)$ in $\L(\Sigma)$ such that $\Lambda(0)=\Lambda$, and denote by $\dot{\Lambda}\in T_{\Lambda}\L(\Sigma)$ its tangent vector. For any point $z\in \Lambda$ and any smooth extension $z(t)\in \Lambda(t)$, we define the quadratic form
\begin{equation} 
\dot{\Lambda}:=  z \mapsto \omega(z,\dot{z}), 
\end{equation} 
where $\dot{z}= \dot{z}(0)$.
A simple check shows that the definition does not depend on the extension $z(t)$. Finally, if in local coordinates $\Lambda(t)=\{(p,S(t)p)|\,p\in \R^{n}\}$, the quadratic form $\dot{\Lambda}$ is represented by the matrix $\dot S(0)$. In other words, if $z \in \Lambda$ has coordinates $p \in \mathbb{R}^n$, then $\dot{\Lambda}[z] = p^*\dot{S}(0)p$.

\subsection{Transversality properties}
In this section we introduce some important properties of curves in the Lagrange Grassmannian. Then we discuss the specific case of a Jacobi curve. Let $J(\cdot)\in \L(\Sigma)$ be a smooth curve in the Lagrange Grassmannian. For $i \in \mathbb{N}$, consider 
\begin{equation}
J^{(i)}(t)=\tx{span}\left\{\frac{d^{j}}{dt^{j}}\ell(t)\bigg| \ \ell(t)\in J(t),\, \ell(t) \text{ smooth},\, 0\leq j \leq i\right\}\subset \Sigma, \qquad i\geq 0.
\end{equation}
\begin{def} \label{d:amplestar}
The subspace $J^{(i)}(t)$ is the \emph{i-th extension} of the curve $J(\cdot)$ at $t$. The flag
\begin{equation}
J(t) = J^{(0)}(t)\subset J^{(1)}(t)\subset J^{(2)}(t)\subset \ldots\subset \Sigma,
\end{equation}
is the \emph{associated flag of the curve} at the point $t$. The curve $J(\cdot)$ is called:
\begin{itemize}
\item[(i)] \emph{equiregular} at $t$ if $\text{dim }J^{(i)}(\cdot)$ is locally constant at $t$, for all $i \in \N$,
\item[(ii)] \emph{ample} at $t$ if there exists $N\in \N$ such that $J^{(N)}(t)=\Sigma$,
\item[(iii)] \emph{monotone increasing} (resp. \emph{decreasing}) at $t$ if $\dot{J}(t)$ is non-negative definite (resp. non-positive definite) as a quadratic form.
\end{itemize}
\end{def}

In coordinates, $J(t)=\{(p,S(t)p)|\  p\in \R^{n}\}$ for some smooth family of symmetric matrices $S(t)$. The curve is ample at $t$ if and only if there exists $N\in \N$ such that 
\begin{equation}
\text{rank} \{\dot S(t), \ddot S(t),\ldots, S^{(N)}(t)\}=n.
\end{equation}
We say that the curve is equiregular, ample or monotone (increasing or decreasing) if it is equiregular, ample or monotone for all $t$ in the domain of the curve.

A crucial property of ample, monotone curves is described in the following lemma.
\begin{lem}\label{l:ampletrasv}
Let $J(\cdot) \in \L(\Sigma)$ a monotone, ample curve at $t_0$. Then, for any fixed Lagrangian subspace $\Lambda$, there exists $\eps > 0$ such that $J(t) \cap \Lambda = 0$ for $0<|t - t_0| <\eps$.
\end{lem}
In other words, ample, monotone curves can intersect any fixed Lagrangian subspace $\Lambda$ only at a discrete set of times.
\begin{proof}
Without loss of generality, assume $t_0 =0$. Choose a Lagrangian splitting $\Sigma = \Lambda \oplus \Pi$, such that, for $|t|<\eps$, the curve is contained in the chart defined by such a splitting. In coordinates, $J(t)=\{(p,S(t)p)|\  p\in \R^{n}\}$, with $S(t)$ symmetric. The curve is monotone, then $\dot{S}(t)$ is a semidefinite symmetric matrix. Without loss of generality, we assume $\dot{S}(t) \geq 0$.
Assume that $J(0) \cap \Lambda \neq 0$. In coordinate, this means that $S(0)$ has some vanishing eigenvalues. We now show that the whole spectrum of $S(t)$ is strictly increasing in $t$, hence it moves away from zero for $t$ sufficiently small.

Notice that $S(t) - S(0) = \int_{0}^t \dot{S}(\tau)d\tau \geq 0$, by the monotonicity assumption. Then, for any $z \in \R^n$, consider the smooth function $t \mapsto f_z(t):=z^*[S(t)-S(0)]z$, which is non-decreasing and vanishes at $t=0$. Moreover, $f_z(t)$ cannot be constantly zero on any interval of the form $[0,\delta)$, otherwise $z$ would be in the kernel of $S(t)-S(0)$ for all $t \in [0,\delta)$ and, a fortiori, in the kernel of all the derivatives $S^{(N)}(0)$, which is absurd by the ampleness hypothesis. Therefore, $f_z(t) >0$ for $0<t<\varepsilon$. Since $z$ is arbitrary
%
%
%
\begin{equation}\label{eq:monotonicity}
S(t) > S(0), \qquad 0<t < \eps.
\end{equation}
Now, denote by $\lambda_1(t)\geq \ldots\geq\lambda_n(t)$ the eigenvalues of $S(t)$ at each fixed $t$. Then, by the Courant min-max principle, we have the following variational characterisation
\begin{equation}
\lambda_k(t) = \max\{\min\{x^*S(t)x|\, x \in U \subset \R^n,\, |x| =1 \}|\, \dim U = k\}, \qquad k=1,\ldots,n.
\end{equation}
Thus, by Eq.~\eqref{eq:monotonicity}, each eigenvalue is strictly increasing for $0<t<\eps$. So, even if $S(0)$ has a non-trivial kernel, it becomes non-degenerate for sufficiently small small $t>0$. The same argument shows that this is true also for $t<0$.
\end{proof}

Observe that, if $\Lambda = J(0)$, then $S(0) = 0$ in any chart given by the splitting $\Sigma = J(0) \oplus \Pi$. Therefore, the proof of Lemma~\ref{l:ampletrasv} implies that all the eigenvalues of $S(t)$ are strictly non-zero for all $|t|<\eps$, $t\neq 0$. If the curve is also monotone and ample, the only restriction on $\eps$ comes from the fact that $J(t)$ must belong to the given coordinate chart. In particular, the eigenvalues of $S(t)$ are strictly positive for all $t>0$ (and strictly negative for $t<0$) at least until the first intersection of $J(\cdot)$ with $\Pi$ occurs. This means that $J(\cdot)$ cannot have further intersections with $J(0)$ until it crosses $\Pi$. Thus, we obtain the following.

\begin{cor}\label{c:ampletrasv}
Let $J(\cdot) \in \L(\Sigma)$ a monotone, ample curve, such that $J(\cdot) \cap \Pi = 0$, for some Lagrangian subspace $\Pi$. Then $J(\cdot)$ has no self-intersections, namely $J(t_1) \cap J(t_2) = 0$ for all $t_1 \neq t_2$.
\end{cor}

\subsection{Reduction}\label{s:red}
Let $(\Sigma,\omega)$ be a symplectic vector space, and let $\Gamma \subset \Sigma$ be an isotropic subspace, namely $\omega|_{\Gamma}\equiv 0$. For any subspace $V \subset \Sigma$, we denote by the symbol $V^\angle$ the corresponding $\omega$-orthogonal subspace.
\begin{definition}
The reduction of $(\Sigma,\omega)$ with respect to an isotropic subspace $\Gamma$ is the symplectic space $(\Sigma^\Gamma,\omega)$, where
\begin{equation}
\Sigma^\Gamma:=\Gamma^\angle / \Gamma.
\end{equation}
\end{definition}
The definition is well posed, since $\omega$ descends to a well-defined symplectic form on the quotient. Moreover, if $\dim\Sigma = 2n$ and $\dim\Gamma = k$, then $\Sigma^\Gamma$ is a $2(n-k)$-dimensional symplectic space.

The projection $\pi^\Gamma : \L(\Sigma) \to \L(\Sigma^\Gamma)$, defined by $\Lambda \mapsto \Lambda \cap \Gamma^\angle /\Gamma$, is not even continuous in general. Nevertheless, the following lemma holds true.
\begin{lem}\label{l:smoothproj}
The restriction of $\pi^\Gamma$ to $\Gamma^\pitchfork := \{\Lambda \in \L(\Sigma)|\, \Lambda\cap\Gamma = 0\}$ is smooth.
\end{lem}
\begin{proof}
Let $\Lambda \in \Gamma^\pitchfork$. We can always find a Lagrangian space $\Pi$ which contains $\Gamma$ and such that $\Pi\cap \Lambda = 0$. The proof is now trivial in charts given by a Darboux basis on the splitting $\Pi \oplus \Lambda$. Indeed, in these charts, the projection corresponds to take a $n-k \times n-k$ block of the representative matrix.
\end{proof}
The next lemma provides condition under which a monotone, ample Jacobi curve remains monotone and ample upon projection.
\begin{lem}\label{l:ampleproj}
Let $J(\cdot) \in \L(\Sigma)$ a monotone, ample (at $t_0$) curve such that $J(\cdot) \in \Gamma^\pitchfork$. Then the projection $J^\Gamma(\cdot):=\pi^\Gamma(J(\cdot))$ is a monotone, ample (at $t_0$) curve in $\L(\Sigma^\Gamma)$.
\end{lem}
\begin{proof}
By Lemma~\ref{l:smoothproj}, the projection $J^\Gamma(\cdot)$ is still smooth. We prove the Lemma by analysing the coordinate presentation of the curve. Without loss of generality, we choose $t_0 = 0$. We find $\Pi \in \L(\Sigma)$ such that $\Gamma \subset \Pi$, and $\Sigma = J(0)\oplus \Pi$. Therefore, we introduce Darboux coordinates $(p,x) \in \R^{2n}$ such that, for small $t$
\begin{equation}
\Pi = \{(0,x)|\,x\in\R^n\}, \qquad J(t) = \{(p,S(t)p)|\,p\in \R^n\}, \qquad S(0) =0.
\end{equation}
Moreover, if $\dim \Gamma = k$, we split $\R^n = \R^k \oplus \R^{n-k}$, and we write $x = (x_1,x_2)$ and $p=(p_1,p_2)$. Thus
\begin{equation}
\Gamma = \{((0,0),(x_1,0))|\,x_1 \in \R^k\}, \qquad \Gamma^\angle = \{((0,p_2),(x_1,x_2))|\, p_2,x_2 \in \R^{n-k}, \, x_1 \in \R^k\}.
\end{equation}
Accordingly, the matrix $S(t)$ splits as $\left(\begin{smallmatrix} S_{11}(t) & S_{12}(t) \\ S_{12}^*(t) & S_{22}(t) \end{smallmatrix}\right)$. In terms of these coordinates, and analogous coordinates on $\Sigma^\Gamma = \Gamma^\angle / \Gamma$, we obtain that the matrix representing the reduced curve is $S^\Gamma(t) := S_{22}(t)$ (which is a $n-k\times n-k$ symmetric matrix). More precisely
\begin{equation}
J^\Gamma(t) = \{(p_2,S_{22}(t) p_2)|\, p_2 \in \R^{n-k}\}.
\end{equation}
The original curve is monotone (say non-decreasing), then $\dot{S}(t) \geq 0$. Therefore, also $\dot{S}_{22}(t) \geq 0$, and $J^\Gamma(\cdot)$ is monotone too. 

We now prove that the reduced curve is still ample at $0$ if the original curve was. We assume $S(t)$ to be real-analytic, otherwise, it is sufficient to replace $S(t)$ with its Taylor polynomial of sufficiently high order. From the proof of Lemma~\ref{l:ampletrasv}, $S(t) > S(0) =0$ for $t>0$ sufficiently small. Thus, for all $y \in \R^{n-k}$, the function $t \mapsto y^*S_{22}(t)y$ is zero at $t=0$, and strictly positive for $t>0$. But an analytic function with these properties has at least a non-vanishing (strictly positive) derivative. Hence, for some $i >0$, $y^*S^{(i)}_{22}(0)y > 0$. Since this construction holds for any $y \in \R^{n-k}$, this implies
\begin{equation}
\rank\{\dot{S}_{22}(0),\ldots,S_{22}^{(N)}(0)\} = n-k,
\end{equation}
for some sufficiently large $N >0$.
\end{proof}

\subsection{Maslov index and conjugate times}
In this section we review a very useful homotopy invariant of curves in the Lagrange Grassmannian: the Maslov index, that is the intersection number of a curve with a certain pseudo-manifold in $\mathcal{L}(\Sigma)$. There are many things called Maslov index in different contexts, for a modern review we suggest \cite{barilari2013geometry}. Here we follow mainly the approach in \cite{NoteCime} and \cite{SympMeth}.

Let $\Pi\in \mathcal{L}(\Sigma)$, consider the following subset of $\mathcal{L}(\Sigma)$,
\begin{equation}
\mathcal{M}_{\Pi}=\mathcal{L}(\Sigma)\setminus\Pi^\pitchfork=\{\Lambda\in \mathcal{L}(\Sigma)|\,\Lambda\cap \Pi\ne 0\},
\end{equation}
which is called the \textit{train} of the Lagrangian subspace $\Pi$ due to V. Arnold. To see how $\mathcal{M}_{\Pi}$ looks locally, let $\Delta \in \mathcal{L}(\Sigma)$, $\Delta \cap \Pi = 0$. In coordinates induced by the splitting $\Sigma =\Pi\oplus \Delta$,
\begin{equation}
\Delta^\pitchfork = \{(p,Sp)|\,p \in \Pi \simeq \R^n,\,S \in Q(\R^n)\}.
\end{equation}
Therefore, in coordinates,
\begin{equation}
\Delta^{\pitchfork}\setminus\Pi^{\pitchfork}\simeq \{S \in Q(\R^n)|\, \ker S \ne 0\}.
\end{equation}
Hence the intersection of $\mathcal{M}_{\Pi}$ with the coordinate neighbourhood $\Delta^{\pitchfork}$ coincides with the set of all degenerate quadratic forms on $\mathbb{R}^n$. Notice that to a subspace $\Lambda$, which has $k$-dimensional intersection with $\Pi$ there corresponds a form with $k$-dimensional kernel. The set of degenerate forms constitute an algebraic hypersurface in the space of all quadratic forms $Q(\mathbb{R}^n)$.

We want to define the intersection number of a curve in $\mathcal{L}(\Sigma)$ with the train $\mathcal{M}_{\Pi}$. To do this, we need a ``co-orientation'' on $\mathcal{M}_{\Pi}$, so first we start describing its singular locus. We see that a point $\Lambda \in \mathcal{M}_{\Pi}$ is singular if its associated quadratic form has at least two-dimensional kernel, so $\mathcal{M}_{\Pi}$ is an algebraic hypersurface in $\mathcal{L}(\Sigma)$ and its singular locus is an algebraic subset of codimension three in $\mathcal{M}_{\Pi}$. Thus $\mathcal{M}_{\Pi}$ is a pseudo-manifold.

Now, let us define a \emph{canonical co-orientation} of the hypersurface $\mathcal{M}_{\Pi}$ at a non-singular point $\Lambda$, i.e. we indicate the ``positive and negative sides'' of $\mathcal{M}_{\Pi}$ in $\mathcal{L}(\Sigma)$. It is not difficult to see that vectors from $T_{\Lambda}\mathcal{L}(\Sigma)$ corresponding to positive definite and negative definite quadratic forms on $\Lambda$ are not tangent to $\mathcal{M}_{\Pi}$, then we have the following.

\begin{definition}
Let $\Lambda$ be a non-singular point of $\mathcal{M}_{\Pi}$. We consider as positive (negative) that side of $\mathcal{M}_{\Pi}$ towards which the positive (negative) definite elements of $T_{\Lambda}\mathcal{L}(\Sigma)$ are directed.
\end{definition}

We say that a curve $J(\cdot)$ in $\mathcal{L}(\Sigma)$ is in \textit{general position} (with respect to $\mathcal{M}_{\Pi}$) if $J(\cdot)$ intersects the non-singular locus of $\mathcal{M}_{\Pi}$ smoothly and transversally. The above co-orientation permits to define correctly the intersection number (or Maslov index) of a continuous curve in general position, with endpoints outside $\mathcal{M}_{\Pi}$, with the hypersurface $\mathcal{M}_{\Pi}$.

\begin{definition}
Let $J(t)$, $t_0\le t\le t_1$ be a continuous curve in general position in $\mathcal{L}(\Sigma)$ with respect to the train $\mathcal{M}_{\Pi}$ such that $J(t_0),J(t_1)\notin \mathcal{M}_{\Pi}$. The Maslov index $J_{[t_0,t_1]}\cdot\mathcal{M}_{\Pi}$ is the number of points where $J(\cdot)$ intersects $\mathcal{M}_{\Pi}$ in the positive direction minus the number of points where this curves intersects $\mathcal{M}_{\Pi}$ in the negative direction.
\end{definition}

A crucial property of the Maslov index is that it is a homotopy invariant of the curve, indeed a homotopy between curves in general position that leaves fixed the endpoints does not change the Maslov index. The proof of this fact is the same as for usual intersection number of a curve with a closed oriented hypersurface (see e.g. \cite{Milnor}). Notice that, since the singular locus of $\mathcal{M}_{\Pi}$ has codimension three, the generic homotopy moves the curve in general position. Thus, the Maslov index of any curve with endpoints not in $\mathcal{M}_{\Pi}$ is defined by putting the curve in general position.

\begin{definition}
Let $J(t)$, $t_0\le t\le t_1$ be a continuous curve (not necessarily in general position) in $\mathcal{L}(\Sigma)$ such that $J(t_0),J(t_1)\notin \mathcal{M}_{\Pi}$. The Maslov index $J_{[t_0,t_1]}\cdot\mathcal{M}_{\Pi}$ is defined as $J^{\prime}_{[t_0,t_1]}\cdot\mathcal{M}_{\Pi}$, where $J^{\prime}(t)$, $t_0\le t\le t_1$ is any curve in general position homotopic to $J_{[t_0,t_1]}$, with the same endpoints.
\end{definition}

A weak point of the definition of Maslov index is the necessity of putting the curve in general position. This does not look like a very efficient way to compute the intersection number since putting the curve in general position could imply the modification of maybe a nice object, but the fact that the Maslov index is homotopy invariant leads to a very simple and effective way to compute it.

\begin{lem}
Assume that the piece of curve $J_{[t_0, t_1]}$ belongs to the chart $\Delta^{\pitchfork}$, $\Delta\cap\Pi=J(t_0)\cap\Pi=J(t_1)\cap\Pi=0$. Let $S(t_i)$ be the symmetric matrix representing the subspace $J(t_i)$ in coordinates given by the splitting $\Sigma = \Delta \oplus \Pi$, that is $J(t_i) = \{(p,S(t_i)p)|p \in \Pi \simeq \R^n\}$. Thus
\begin{equation}
J_{[t_0, t_1]}\cdot\mathcal{M}_{\Pi}=\ind S(t_0)-\ind S(t_1),
\end{equation}
where $\ind S$ is the index of the quadratic form $z\mapsto z^*Sz$, $z\in \mathbb{R}^n$.
\end{lem}

In general the whole curve is not contained in a chart, but we can split it into segments $J_{[\tau_i, \tau_{i+1}]}$, $i=0,\ldots,\ell$, in such a way that $J(\tau)\in \Delta_i^{\pitchfork}$ $\forall\tau\in [\tau_i, \tau_{i+1}]$, where $\Delta_i\cap\Pi=0$, $i=0,\ldots,\ell$. Hence
\begin{equation}
J(\cdot)\cdot\mathcal{M}_{\Pi}=\sum_{i=0}^\ell J_{[\tau_i, \tau_{i+1}]}\cdot\mathcal{M}_{\Pi}.
\end{equation}

\begin{rem}
In particular, if $J(\cdot)$ is a Jacobi curve (which is monotone and ample) then the absolute value of the Maslov index $J_{[t_0,t_1]}\cdot\mathcal{M}_{\mathcal{V}}$ is the number of conjugate times of $J(\cdot)$ counted with multiplicity in the interval $[t_0,t_1]$.
\end{rem}

We finish this section with the two main propositions about the Maslov index we need in the following. The first one provides an estimate of the difference of the Maslov index of a curve with respect to two different trains.
\begin{prop}\label{p:change1}
 Let $J(t)$, $t_0\le t\le t_1$, be a continuous curve in $\mathcal{L}(\Sigma)$ and suppose that $\Pi,\Pi'\in \mathcal{L}(\Sigma)$ satisfy $\Pi \cap J(t_i)=\Pi'\cap J(t_i)=0$, $i=0,1$. Then
\begin{equation}
\lvert J_{[t_0,t_1]}\cdot \mathcal{M}_{\Pi} - J_{[t_0,t_1]}\cdot \mathcal{M}_{\Pi'} \rvert \le n.
\end{equation}
\end{prop}
In the second one we consider a continuous curve $P_t$ in $Sp(\Sigma)$, i.e. a one-parameter subgroup of the group $Sp(\Sigma)$ of \textit{symplectic transformations} of $\Sigma$ and we estimate the difference of the indices between two curves generated by $P_t$ with respect to the same train.
\begin{prop}\label{p:change2}
Let $P_t\in Sp(\Sigma)$, $t_0\le t \le t_1$ be a continuous curve in $Sp(\Sigma)$, $P_{t_0}=\mathbb{I}$, and suppose $\Lambda,\Lambda^\prime\in \mathcal{L}(\Sigma)$. Set $J(t)=P_{t}\Lambda$ and $J^\prime(t)=P_{t}\Lambda^\prime$. Then, for all $\Pi \in L(\Sigma)$ such that $\Pi\cap J(t_i)=\Pi\cap J^\prime(t_i)=0$, $i=0,1$, the following inequality holds
\begin{equation}
\lvert J_{[t_0,t_1]}\cdot \mathcal{M}_{\Pi} - J^\prime_{[t_0,t_1]}\cdot \mathcal{M}_{\Pi} \rvert \le n.
\end{equation}
\end{prop}
The proofs of Propositions~\ref{p:change1} and~\ref{p:change2} can be found in \cite[Propositions 5, 6]{SympMeth}.

\section{Main results}\label{s:main}

We start this section by defining, more precisely, the class of dynamical systems under investigation. Let $(\Sigma,\sigma)$ be a symplectic vector space.
\begin{definition}\label{d:LQ}
A LQ optimal control problem is a pair $(H,\ve)$, where $H : \Sigma \to \R$ is a quadratic form (the Hamiltonian) and $\ve \subset \Sigma$ is a Lagrangian subspace, such that $H|_\ve \geq 0$.
\end{definition}
By choosing appropriate Darboux coordinates, $\Sigma = \R^{2n}$, $\omega = \J$, $\ve = \{(p,0)|\,p \in \R^{n}\}$ and the Hamiltonian is
\begin{equation}\label{eq:Hamiltonian2}
H(p,x) = \frac{1}{2}(p,x)^* \H \begin{pmatrix}
p \\ x
\end{pmatrix}, \qquad \H = \begin{pmatrix}
BB^* & A \\ A^* & Q
\end{pmatrix}.
\end{equation}
Thus, Definition~\ref{d:LQ} is a coordinate-free characterization of the systems introduced in Sec.~\ref{s:intro}. 
With the pair $(H,\ve)$ we associate the Jacobi curve $J(t) = e^{t\J\H}\ve$, which is a smooth curve in the Lagrange Grassmannian $\L(\Sigma)$. The assumption $H|_\ve \geq 0$ is equivalent to the monotonicity of $J(\cdot)$.
\begin{lem}
The Jacobi curve of the system $(H,\ve)$ is monotone and equiregular.
\end{lem}
\begin{proof}
Let $z \in J(t)$, then there exists $z_0 \in \ve$ such that $z = e^{t\J\H}z_0$. The last formula also provides a smooth extension of $z$ belonging to the Jacobi curve for times close to $t$. Then, by definition of the quadratic form $\dot{J}(t)$, we obtain
\begin{equation}
\dot{J}(t)[z] = \omega(z,\dot{z}) = \omega\left(e^{t\J\H}z_0,e^{t\J\H}\J\H z_0\right) = \omega(z_0,\J\H z_0) = -z_0^* BB^* z_0 \leq 0,
\end{equation}
where we have used the fact that the Hamiltonian flow is a one-parameter group of symplectomorphisms. This proves that $\dot{J}(t) \leq 0$ as a quadratic form and the curve is monotone.

Now observe that $J(t+\eps) =  e^{t\J\H} J(\eps)$. This imples, by definition of $i$-th extension, that
\begin{equation}
J^{(i)}(t)  = e^{t\J\H} J^{(i)}(0), \qquad i \geq 0,
\end{equation}
hence the $i$-th extensions have the same dimension for all $t$, and the curve is equiregular. 
\end{proof}
\begin{lem}
The system $(H,\ve)$ is controllable if and only if the Jacobi curve $J(\cdot)$ is ample.
\end{lem}
\begin{proof}
By definition, the system $(H,\ve)$, which can be written as in Eq.~\eqref{eq:Hamiltonian2}, is controllable if 
\begin{equation}
\rank(B,AB,\ldots,A^{m-1}B) = n.
\end{equation}
It is sufficient to prove that this is equivalent to ampleness at $t=0$, since ampleness at all $t$ follows from the equiregularity of the curve. Indeed, for small $t$, $J(t) = \{(p,S(t)p)|\,p \in \R^n\}$. We explicitly compute $S(t)$ as follows. Observe that
\begin{equation}
J(t) = e^{t\J\H}\begin{pmatrix} p \\ 0 \end{pmatrix} = \begin{pmatrix}
\phi_{11}(t) & \phi_{12}(t) \\ \phi_{21}(t) & \phi_{22}(t)
\end{pmatrix} \begin{pmatrix} p \\ 0 \end{pmatrix}, \qquad p \in \R^n.
\end{equation}
It is clear that $S(t) = \phi_{21}(t) \phi_{11}(t)^{-1}$. Then we can compute iteratively the derivatives of $S(t)$ at $t=0$, and we obtain, for any $m >0$
\begin{equation}
\rank\{\dot{S}(0),\ddot{S}(0),\ldots,S^{m-1}(0)\} = \rank\{B,AB,\ldots,A^{m-1}B\}.
\end{equation}
Therefore controllability is equivalent to ampleness of the curve at $t=0$.
\end{proof}

We employ the symbol $\HH$ to denote the set of controllable dynamical systems $(H,\ve)$ or, with no risk of confusion, the associated Hamiltonian vector fields $\vec{H}$. Since the associated Jacobi curve is monotone, ample and equiregular, Lemma~\ref{l:ampletrasv} and Corollary~\ref{c:ampletrasv} apply. This has important consequences on conjugate times.

\begin{definition}
We say that $\Gamma\subset \Sigma$ is an $\vec{H}$-invariant subspace if $P_t( \Gamma )= \Gamma$ for all $t \in \R$.
\end{definition}

\begin{prop} \label{lem_lagrangian_invariant_no_conj_points}
Let $\vec{H} \in \HH$. Suppose there exists an $\vec{H}$-invariant Lagrangian subspace $\Gamma \subset \Sigma$, then the Jacobi curve $J(\cdot)$ has no conjugate times.
\end{prop}
\begin{proof}
Indeed, by Lemma~\ref{l:ampletrasv}, the Jacobi curve remains transversal to $\Gamma$ for all times. Then, by Corollary~\ref{c:ampletrasv}, the only intersection with $\ve = J(0)$ can occur at $t=0$.
\end{proof}
Notice that the Lagrangian hypothesis is crucial. Indeed, Proposition~\ref{lem_lagrangian_invariant_no_conj_points} is false if the $\vec{H}$-invariant subspace is simply isotropic.

\subsection{Proof of the main result}
Now we are ready to prove Theorem~\ref{t:main}. By ``eigenvalues of the Hamiltonian'' we will mean the eigenvalues of $\J\H$, that is the matrix representing the Hamiltonian vector field $\vec{H}$. The proof is based on the following steps:
\begin{itemize}
\item[(i)] Assuming $\vec{H}$ diagonalizable, with pure imaginary spectrum, there are infinitely many conjugate times (Proposition~\ref{prop_pure_imaginary_spectrum});
\item[(ii)] Assuming $\vec{H}$ diagonalizable, with at least one pure imaginary eigenvalue, there are infinitely many conjugate times (Proposition~\ref{prop_at_least_one_pure_imag}).
\item[(iii)] For a general $\vec{H}$, with at least one Jordan block of odd order corresponding to a pure imaginary eigenvalue, there are infinitely many conjugate times (Proposition~\ref{prop_oddjordan}).
\end{itemize}
\begin{itemize}
\item[(iv)] For a general $\vec{H}$, if all Jordan blocks corresponding to pure imaginary eigenvalues are of even order, there are no conjugate times (Proposition~\ref{prop_noconjtimes}).
\end{itemize}
We directly prove (i). Then, with the techniques of Sec.~\ref{s:red}, we reduce (ii) and (iii) to the ``extremal'' case (i).

We start by recalling an important property of the spectrum of Hamiltonian matrices as $\vec{H}$. If $\lambda$ is an eigenvalue, then also $\pm\lambda, \pm \bar{\lambda}$ are eigenvalues with the same multiplicity, where the bar denotes complex conjugation. Then, eigenvalues always appear in pairs (if $\lambda = \beta$ or $\lambda = i\beta$ for $\beta \in \R$) or in quadruples otherwise.

We denote by $E_{\lambda} \subseteq \R^{2n}$ the real invariant subspace corresponding to the eigenvalues $\lambda,\bar{\lambda}$ of $\vec{H}$. This is the real vector space generated by the generalized eigenvectors $\xi$, $\bar{\xi}$ corresponding to the eigenvalues $\lambda$ and $\bar\lambda$, respectively. More precisely
\begin{equation}
E_{\lambda}:= \spn\{u,v \in \R^{2n} |\, u + iv \in \ker(\vec{H}-\lambda\mathbb{I})^k,\,k \geq 0\}.
\end{equation}
It is clear that $E_{\lambda} = E_{\bar{\lambda}}$. 
\begin{lem} \label{lem_H_orthog}
Let $\lambda$ and $\lambda^\prime$ be eigenvalues of $\vec{H}$ (not necessarily distinct). If $\lambda + \lambda^\prime \neq 0$ and $\bar\lambda + \lambda^\prime \neq 0$ then $E_{\lambda}\J E_{\lambda^\prime}=0$.
\end{lem}
\begin{proof}
For simplicity, we prove the theorem assuming $\vec{H}$ to be diagonalizable. Recall that $\vec{H} = -\J\H$ and $\Omega^2 = -\mathbb{I}$. Let $\xi$ and $\xi^\prime$ be eigenvectors corresponding to $\lambda$ and $\lambda^\prime$ respectively. Since $\J^2 = -\mathbb{I}$, we have $\xi^\prime\mathbf{H}\xi=\lambda\xi^\prime\J\xi$ and $\xi\mathbf{H}\xi^\prime=\lambda^\prime\xi\J\xi^\prime$ so $(\lambda+\lambda^\prime)\xi\J\xi^\prime=0$. Analogously, we obtain $\xi'H \bar{\xi} = \bar{\lambda} \xi' \J \bar\xi$ and $\bar{\xi}H\xi' = \lambda'\bar\xi\J\xi'$. Then $(\bar\lambda+\lambda^\prime)\bar\xi\J\xi^\prime=0$. Since $\lambda + \lambda^\prime \ne 0$ and $\bar\lambda + \lambda^\prime \ne 0$ it follows that $E_{\lambda}\J E_{\lambda^\prime}=0$.
The above result still holds if $\vec{H}$ is not diagonalizable (see \cite[Lemma D.1, Chapter II]{MeyerHall}).
\end{proof}
\begin{rem}
In particular if $\lambda=\alpha+i\beta$, with $\alpha\ne 0$ then $\J\rvert_{E_\lambda}\equiv 0$, i.e. $E_{\alpha+i\beta}$ is isotropic if $\alpha \neq 0$.
\end{rem}
It follows that the invariant subspaces associated with purely imaginary eigenvalues, non-purely imaginary eigenvalues, and $E_0$ are pairwise $\J$-orthogonal. This, together with the non-degeneracy of $\J$, implies the following decomposition in $\J$-orthogonal symplectic subspaces
\begin{equation}
\R^{2n} =  \underbrace{E_0 \oplus \left(\bigoplus_{\alpha\neq 0} E_{\alpha+i\beta}\right)}_{\text{non pure imaginary}} \oplus\underbrace{\left( \bigoplus_{\beta \neq 0} E_{i\beta} \right)}_{\text{pure imaginary}}.
\end{equation}
In the following, with the term ``pure imaginary eigenvalue'' we understand all the eigenvalues $\lambda = i\beta$, with $\beta \neq 0$.
\begin{lem}\label{lem_lagrangian_invariant}
There exists an $\vec{H}$-invariant, Lagrangian subspace $\Gamma_+$ of the symplectic space $\displaystyle{\bigoplus_{\substack{\lambda \text{ non pure} \\ \text{imaginary}}}E_\lambda}$.
\end{lem}

\begin{proof}
If zero is not an eigenvalue of $\vec{H}$, we take $\Gamma_+=\displaystyle{\bigoplus_{\alpha>0}}E_{\alpha+i\beta}$, which is $\vec{H}$-invariant by definition. If zero is an eigenvalue of $\vec{H}$, let us consider the corresponding invariant subspace $E_0$, with $\dim E_0=2m$. Choose an isotropic $m$-dimensional subspace $\Gamma_0\subset E_0$ (which is indeed $\vec{H}$-invariant). Hence $\Gamma_+=\Gamma_0\oplus\displaystyle{\bigoplus_{\substack{\alpha>0}}}E_{\alpha+i\beta}$ satisfies the required properties.
\end{proof}

%

\subsubsection{Diagonalizable case}

In this section, we assume $\vec{H}$ to be diagonalizable.

\begin{prop} \label{prop_pure_imaginary_spectrum} 
Let $\vec{H} \in \HH$. Suppose that $\vec{H}$ is diagonalizable and has a pure imaginary spectrum. Then the Jacobi curve $J(\cdot)$ has infinitely many conjugate times.
\end{prop}
\begin{proof}
If $\vec{H}$ has only pure imaginary eigenvalues, it is well known (see e.g. \cite[Appendix 6]{MathMethCM}) that there exists a symplectic change of coordinates such that the Hamiltonian can be written as 
\begin{equation}\label{eq:normalpureim}
H(p,x)=\frac{1}{2}\sum_{j=1}^n \omega_j(p_j^2+x_j^2),\qquad \omega_1\ge \omega_2\ge \dots \ge \omega_n. 
\end{equation}
Notice that the eigenvalues of $\vec{H}$ are $\pm i\omega_j$, $j=1,\dots,n$. The signs of the $\omega_j$ are precisely the signs of $H$ on the real eigenspaces $E_{i\omega_j}$. The following two lemmas are crucial.
\begin{lem}[Givental' \cite{Givental}]\label{lem_cond_omega_j+omega_{n-j+1}}
There exists a Lagrangian subspace $\Lambda\subset \R^{2n}$ such that $H\rvert_\Lambda >0$ if and only if $\omega_j+\omega_{n-j+1}>0$, $j=1,\dots,n$.
\end{lem}
\begin{lem}[Fa{\u\i}busovich \cite{ExUnRiccEq}]\label{lem_controllable_admits_positive_lag}
Under the controllability assumption (or, equivalently, the ampleness of the Jacobi curve), there exists a Lagrangian subspace $\Lambda \subset \R^{2n}$ such that $H\rvert_\Lambda>0$.
\end{lem}
Lemmas~\ref{lem_controllable_admits_positive_lag} and~\ref{lem_cond_omega_j+omega_{n-j+1}} imply the following inequality:
\begin{equation}\label{eq:inequality}
\sum_{j=1}^n \omega_j>0.
\end{equation}
Now, let us define a new curve $L(t):=P_t (L_0)$ in $\L(\R^{2n})$, where $L_0:=\{(p,0):p\in \mathbb{R}^n\}\subset \R^{2n}$, $L_0\in \L(\R^{2n})$.
\begin{rem}
Notice that, in order to bring the Hamiltonian to the normal form of Eq.~\eqref{eq:normalpureim}, we have done a symplectic change of basis. Thus, in general, $L_0 \neq \ve$.
\end{rem}
If we reorder coordinates in such a way that $(p,x) \mapsto (p_1,x_1,\ldots,p_n,x_n)$, we can write 
\begin{equation}
L(t)=\begin{pmatrix}r(t\omega_1) &  &  \\ & \ddots & \\ & & r(t\omega_n)\end{pmatrix}L_0,
\end{equation}
where $r(t\omega_j)$ is a rotation of angle $t\omega_j$ in the $2$-dimensional subspace $(p_j,x_j)$. Observe that, given $t>0$ we can choose $\varepsilon>0$ sufficiently small such that $L(\varepsilon)\cap L_0=L(t+\varepsilon)\cap L_0= 0$. Therefore the Maslov index $L_{[\varepsilon,t+\varepsilon]}\cdot \mathcal{M}_{L_0}$ is well defined, since the endpoints of the curve are transversal to the train. 
We employ the shorthand $L_{(0,t)}\cdot\mathcal{M}_{L_0} = L_{[\varepsilon,t+\varepsilon]}\cdot \mathcal{M}_{L_0}$, for any $\varepsilon$ sufficiently small, and similar notation is understood every time a small variation of the end-times is required.

We now prove that the index $L_{(0,+\infty)}\cdot\mathcal{M}_{L_0}$ is infinite. Intersections with the train occur at each half-rotation in each $2$-dimensional subspace $(p_j,x_j)$, with a sign given by the sign of $\omega_j$. Therefore, by a direct computation, we have
\begin{equation}
L_{(0,T)}\cdot\mathcal{M}_{L_0} = \sum_{j=1}^{n}\lfloor \frac{T\omega_j}{\pi}\rfloor > \sum_{j=1}^{n}\frac{T\omega_j}{\pi}-n.
\end{equation}
Inequality~\eqref{eq:inequality} implies that there are no compensations in the sum of the signs in the computation of the Maslov index. Indeed, let $N>0$ fixed. Since $\sum_{j=1}^n \omega_j>0$ we can take $T \ge \frac{(N+n)\pi}{\sum_{j=1}^{n}\omega_j}$ so that
\begin{equation}\label{ineq_conj_time}
L_{(0,T)}\cdot\mathcal{M}_{L_0} > N.
\end{equation}
This implies that the Maslov index of the curve $L(t) = P_t (L_0)$ with the train $\mathcal{M}_{L_0}$ grows to infinity for $T \to \infty$. On the other hand, the number of conjugate times (counted with multiplicity) is the Maslov index of the Jacobi curve $J(t) = P_t (\ve)$ with the train $\mathcal{M}_{\ve}$. Thus, by combining Proposition~\ref{p:change1} and~\ref{p:change2}, we obtain
\begin{equation}
\lvert J_{(0,T)}\cdot \mathcal{M}_{\ve} - L_{(0,T)}\cdot \mathcal{M}_{L_0} \rvert \le 2n.
\label{ineq_prop5&6}
\end{equation}
Therefore
\begin{equation}
J_{(0,T)}\cdot\mathcal{M}_\ve > \frac{\sum_{j=1}^n\omega_j}{\pi}T-3n.
\end{equation}
Thus $J(\cdot)$ has infinitely many conjugate times.
\end{proof}

As a corollary of the proof of Proposition~\ref{prop_pure_imaginary_spectrum}, we can give an estimate for the first conjugate time of a LQ optimal control problem.
\begin{cor}\label{c:estimate1}
Suppose the Hamiltonian can be written as $H(p,x)=\frac{1}{2}\sum_{j=1}^n \omega_j(p_j^2+x_j^2)$. Then, if $T \geq \frac{(N+3n-1)\pi}{\sum_{j=1}^n\omega_j}$ there are at least $N$ conjugate times (counted with multiplicity) in the interval $(0,T]$. In particular, the first conjugate time $\bar{t}$ satisfies $\bar{t}\le \frac{3n\pi}{\sum_{j=1}^{n}\omega_j}$.
\end{cor}

Now we are ready to discuss the case in which both pure and non pure imaginary eigenvalues occur in the spectrum of $\vec{H}$.
\begin{prop}\label{prop_at_least_one_pure_imag} 
Let $\vec{H} \in \HH$. Assume that $\vec{H}$ is diagonalizable and has at least one pure imaginary eigenvalue. Then the associated Jacobi curve $J(\cdot)$ has infinitely many conjugate times.
\end{prop}

\begin{proof}
We reduce the problem to the extremal case of Proposition~\ref{prop_pure_imaginary_spectrum}. Consider $\Gamma_+$ as in Lemma~\ref{lem_lagrangian_invariant}, and let $\dim\Gamma=k$ (we drop the index $+$ from now on). Recall that $\Gamma$ is an $\vec{H}$-invariant isotropic subspace of $\Sigma=\R^{2n}$. We will consider the Lagrange Grassmannian of the reduced space $\Sigma^\Gamma=\Gamma^\angle/\Gamma$. Notice that, by Lemma~\ref{l:ampletrasv}, the Jacobi curve remains transversal to $\Gamma$ for all times. Thus, by Lemma~\ref{l:ampleproj}, the reduced Jacobi curve $J^\Gamma(\cdot)$ is a smooth, ample, monotone curve in $\L(\Sigma^\Gamma)$. By construction, we have
\begin{equation}
\Gamma^\angle=\Gamma \oplus \displaystyle{\bigoplus_{\substack{\lambda \text{ pure} \\ \text{imaginary}}}E_\lambda}, \qquad \Sigma^\Gamma=\displaystyle{\bigoplus_{\substack{\lambda \text{ pure} \\ \text{imaginary}}}E_\lambda}.
\end{equation}
Therefore we reduced the problem to the case of purely imaginary spectrum, and we can apply Proposition~\ref{prop_pure_imaginary_spectrum} to conclude that $J^\Gamma(\cdot)$ has infinitely many conjugate times. Notice that conjugate times for $J^\Gamma(\cdot)$ are intersections with $\ve^\Gamma:=\pi(\ve) = (\Gamma^\angle \cap\ve) / \Gamma$. This means that the original curve has infinitely many intersections with $\ve^\Gamma\oplus \Gamma$. More precisely, as we obtained in the proof of Proposition~\ref{prop_pure_imaginary_spectrum}, and recalling that $\dim \Sigma^\Gamma = 2(n-k)$ we have
\begin{equation}
J_{(0,T)} \cdot\mathcal{M}_{\ve^\Gamma\oplus\Gamma} >\frac{\sum_{j=1}^{n-k}\omega_j}{\pi}T - 3(n-k).
\end{equation}
By applying again Proposition~\ref{p:change1}, we obtain
\begin{equation}
\lvert J_{(0,T)}\cdot \mathcal{M}_{\ve} - J_{(0,T)}\cdot \mathcal{M}_{\ve^\Gamma\oplus \Gamma} \rvert \le n.
\end{equation}
Therefore
\begin{equation}
J_{(0,T)} \cdot\mathcal{M}_{\ve} >\frac{\sum_{j=1}^{n-k}\omega_j}{\pi}T - 4n +3k.
\end{equation}
Then $J(\cdot)$ has infinitely many conjugate times as well.
\end{proof}

Again, we give an estimate for the number of conjugate times as a separate corollary.

\begin{cor}\label{c:estimate2}
Suppose the Hamiltonian, restricted to $\displaystyle{\bigoplus_{\substack{\lambda \text{ pure} \\ \text{imaginary}}}E_\lambda}$, can be written as $H(p,x)=\frac{1}{2}\sum_{j=1}^{n-k} \omega_j(p_j^2+x_j^2)$. Then if $T \geq \frac{(N+4n-3k-1)\pi}{\sum_{j=1}^{n-k}\omega_j}$, there are at least $N$ conjugate times (counted with multiplicity) in the interval $(0,T]$. In particular, the first conjugate time $\bar{t}$ satisfies $\bar{t}\le \frac{(4n-3k)\pi}{\sum_{j=1}^{n-k}\omega_j}$. 
\end{cor}

\subsubsection{General case}

Now, let us consider an arbitrary $\vec{H}$. We approach the problem with the same basic techniques devised for the diagonalizable case. Let $\lambda=i\beta$, $\beta \neq 0$ a pure imaginary eigenvalue of $\vec{H}$. Recall that, by Lemma~\ref{lem_H_orthog}, $E_{i\beta}$ is $\J$-orthogonal to all the others $E_{\lambda'}$, with $\lambda' \neq \pm i \beta$. Therefore $E_{i\beta}$ is symplectic. It is well known that there exists a symplectic change of coordinates on $E_{i\beta}$ such that the Hamiltonian $H\rvert_{E_{i\beta}}$ has one of the following normal forms (see \cite{Ciampi,Williamson} and \cite[Appendix 6]{MathMethCM}).
\begin{enumerate}
\item[(a)] If $\pm i\beta$ correspond to a pair of Jordan blocks of even order $2k$:
 \begin{multline}\label{ham_normal_form_2k}
    H(p,x) = \pm\frac{1}{2}\left[\sum_{j=1}^k\left(\frac{1}{\beta^2} x_{2j-1}x_{2k-2j+1}+x_{2j}x_{2k-2j+2}\right) - \beta^2\sum_{j=1}^{k}p_{2j-1}x_{2j} + \sum_{j=1}^k p_{2j}x_{2j-1} - \right. \\
            - \left.\sum_{j=1}^{k-1}\left(\beta^2 p_{2j+1}p_{2k-2j+1}+p_{2j+2}p_{2k-2j+2}\right)\right].
\end{multline}
\item[(b)] If $\pm\lambda$ correspond to a pair of Jordan blocks of odd order $2k+1$:
\begin{multline}\label{ham_normal_form_2k+1}
    H(p,x) = \pm\frac{1}{2}\left[\sum_{j=1}^k\left(\beta^2 p_{2j}p_{2k-2j+2}+x_{2j}x_{2k-2j+2}\right)\right. -\sum_{j=1}^{2k}p_jx_{j+1} - \\
            - \left.\sum_{j=1}^{k+1}\left(\beta^2 p_{2j-1}p_{2k-2j+3}+x_{2j-1}x_{2k-2j+3}\right)\right].
\end{multline}
\end{enumerate}
Notice that the dimension of $E_\lambda$ is $4k$ or $4k+2$, respectively.

\begin{lem}\label{lem_Jordan_blocks_isotropic_subspace}
 Let $\lambda=i\beta$ a pure imaginary eigenvalue of $\vec{H}$. Thus
\begin{enumerate}
 \item[(a)] If the Jordan block corresponding to $\lambda$ has even order $2k$ then there exists a Lagrangian $\vec{H}$-invariant subspace $\Gamma\subset E_{\lambda}$ (of dimension $2k$).
 \item[(b)] If the Jordan block corresponding to $\lambda$ has odd order $2k+1$ then there exists an isotropic $\vec{H}$-invariant subspace $\Gamma\subset E_{\lambda}$ of dimension $2k$.
\end{enumerate}
\end{lem}
\begin{proof}
Let us consider the first case. As we recall above, $H\rvert_{E_\lambda}$ can be written as in Eq.~\eqref{ham_normal_form_2k}. Then, a careful inspection shows that $\vec{H}|_{E_\lambda} = -\J\H|_{E_\lambda}$ has the structure, in coordinates $(p,x) \in \R^{4k}$, displayed in Fig.~\ref{f:a}.
\begin{figure}
\centering
\subfigure[\label{f:a}]{\begin{tikzpicture}[x=0.30mm, y=0.30mm, inner xsep=0pt, inner ysep=0pt, outer xsep=0pt, outer ysep=0pt]
\path[line width=0mm] (-82.00,-82.00) rectangle +(164.00,164.00);
\definecolor{L}{rgb}{0,0,0}
\definecolor{F}{rgb}{0.827,0.827,0.827}
\path[line width=0.02mm, draw=L, fill=F] (-80.00,60.00) rectangle +(20.00,20.00);
\path[line width=0.02mm, draw=L, fill=F] (-60.00,40.00) rectangle +(20.00,20.00);
\path[line width=0.02mm, draw=L, fill=F] (-20.00,0.00) rectangle +(20.00,20.00);
\path[line width=0.02mm, draw=L, fill=F] (40.00,40.00) rectangle +(20.00,20.00);
\path[line width=0.02mm, draw=L, fill=F] (60.00,60.00) rectangle +(20.00,20.00);
\path[line width=0.02mm, draw=L, fill=F] (0.00,0.00) rectangle +(20.00,20.00);
\path[line width=0.02mm, draw=L, fill=F] (-20.00,-40.00) rectangle +(20.00,20.00);
\path[line width=0.02mm, draw=L, fill=F] (-60.00,-80.00) rectangle +(20.00,20.00);
\path[line width=0.02mm, draw=L, fill=F] (-40.00,20.00) rectangle +(20.00,20.00);
\path[line width=0.02mm, draw=L, fill=F] (20.00,20.00) rectangle +(20.00,20.00);
\path[line width=0.02mm, draw=L, fill=F] (-40.00,-60.00) rectangle +(20.00,20.00);
\draw(-70.00,68.00) node[anchor=base]{\fontsize{6}{7}\selectfont $1$};
\draw(-50.00,48.00) node[anchor=base]{\fontsize{6}{7}\selectfont $2$};
\draw(-10.00,7.00) node[anchor=base]{\fontsize{6}{7}\selectfont $k$};
\path[line width=0.02mm, draw=L, fill=F] (0.00,-20.00) rectangle +(20.00,20.00);
\path[line width=0.02mm, draw=L, fill=F] (20.00,-40.00) rectangle +(20.00,20.00);
\path[line width=0.02mm, draw=L, fill=F] (60.00,-80.00) rectangle +(20.00,20.00);
\path[line width=0.02mm, draw=L, fill=F] (40.00,-60.00) rectangle +(20.00,20.00);
\definecolor{F}{rgb}{0,0,0}
\path[line width=0.05mm, draw=L, fill=F] (-30.00,30.00) circle (0.09mm);
\path[line width=0.05mm, draw=L, fill=F] (-25.00,25.00) circle (0.09mm);
\path[line width=0.05mm, draw=L, fill=F] (30.00,30.00) circle (0.09mm);
\path[line width=0.05mm, draw=L, fill=F] (50.00,-50.00) circle (0.09mm);
\path[line width=0.05mm, draw=L, fill=F] (55.00,-55.00) circle (0.09mm);
\path[line width=0.05mm, draw=L, fill=F] (45.00,-45.00) circle (0.09mm);
\path[line width=0.05mm, draw=L, fill=F] (35.00,35.00) circle (0.09mm);
\path[line width=0.05mm, draw=L, fill=F] (25.00,25.00) circle (0.09mm);
\path[line width=0.05mm, draw=L, fill=F] (-35.00,35.00) circle (0.09mm);
\draw(-50.00,-72.00) node[anchor=base]{\fontsize{6}{7}\selectfont $k$-$1$};
\path[line width=0.05mm, draw=L, fill=F] (-30.00,-50.00) circle (0.09mm);
\path[line width=0.05mm, draw=L, fill=F] (-25.00,-45.00) circle (0.09mm);
\path[line width=0.05mm, draw=L, fill=F] (-35.00,-55.00) circle (0.09mm);
\path[line width=0.05mm, draw=L] (-80.00,-80.00) rectangle +(160.00,160.00);
\path[line width=0.05mm, draw=L] (-80.00,0.00) -- (80.00,0.00);
\path[line width=0.05mm, draw=L] (0.00,-80.00) -- (0.00,80.00);
\draw(10.00,7.00) node[anchor=base]{\fontsize{6}{7}\selectfont $k$};
\draw(50.00,48.00) node[anchor=base]{\fontsize{6}{7}\selectfont $2$};
\draw(70.00,68.00) node[anchor=base]{\fontsize{6}{7}\selectfont $1$};
\draw(-10.00,-32.00) node[anchor=base]{\fontsize{6}{7}\selectfont $1$};
\draw(10.00,-12.00) node[anchor=base]{\fontsize{6}{7}\selectfont $1$};
\draw(30.00,-32.00) node[anchor=base]{\fontsize{6}{7}\selectfont $2$};
\draw(70.00,-72.00) node[anchor=base]{\fontsize{6}{7}\selectfont $k$};
\end{tikzpicture}%
}\qquad\qquad\qquad\qquad
\subfigure[\label{f:b}]{\begin{tikzpicture}[x=0.30mm, y=0.30mm, inner xsep=0pt, inner ysep=0pt, outer xsep=0pt, outer ysep=0pt]
\path[line width=0mm] (-82.00,-82.00) rectangle +(164.00,164.00);
\definecolor{L}{rgb}{0,0,0}
\definecolor{F}{rgb}{0.827,0.827,0.827}
\path[line width=0.02mm, draw=L, fill=F] (-40.00,20.00) [rotate around={270:(-40.00,20.00)}] rectangle +(20.00,20.00);
\path[line width=0.02mm, draw=L, fill=F] (40.00,60.00) [rotate around={270:(40.00,60.00)}] rectangle +(20.00,20.00);
\path[line width=0.02mm, draw=L, fill=F] (60.00,80.00) [rotate around={270:(60.00,80.00)}] rectangle +(20.00,20.00);
\path[line width=0.02mm, draw=L, fill=F] (0.00,20.00) [rotate around={270:(0.00,20.00)}] rectangle +(20.00,20.00);
\path[line width=0.02mm, draw=L, fill=F] (-60.00,40.00) [rotate around={270:(-60.00,40.00)}] rectangle +(20.00,20.00);
\path[line width=0.02mm, draw=L, fill=F] (20.00,40.00) [rotate around={270:(20.00,40.00)}] rectangle +(20.00,20.00);
\draw(-30.00,8.00) node[anchor=base]{\fontsize{6}{7}\selectfont $2k$};
\definecolor{F}{rgb}{0,0,0}
\path[line width=0.05mm, draw=L, fill=F] (-50.00,30.00) circle (0.09mm);
\path[line width=0.05mm, draw=L, fill=F] (-45.00,25.00) circle (0.09mm);
\path[line width=0.05mm, draw=L, fill=F] (30.00,30.00) circle (0.09mm);
\path[line width=0.05mm, draw=L, fill=F] (35.00,35.00) circle (0.09mm);
\path[line width=0.05mm, draw=L, fill=F] (25.00,25.00) circle (0.09mm);
\path[line width=0.05mm, draw=L, fill=F] (-56.00,36.00) circle (0.09mm);
\path[line width=0.05mm, draw=L] (-80.00,80.00) [rotate around={270:(-80.00,80.00)}] rectangle +(160.00,160.00);
\path[line width=0.05mm, draw=L] (-80.00,0.00) -- (80.00,0.00);
\path[line width=0.05mm, draw=L] (0.00,-80.00) -- (0.00,80.00);
\draw(10.00,8.00) node[anchor=base]{\fontsize{6}{7}\selectfont $2k$+$1$};
\draw(50.00,48.00) node[anchor=base]{\fontsize{6}{7}\selectfont $2$};
\draw(70.00,68.00) node[anchor=base]{\fontsize{6}{7}\selectfont $1$};
\definecolor{F}{rgb}{0.827,0.827,0.827}
\path[line width=0.02mm, draw=L, fill=F] (-80.00,-60.00) [rotate around={270:(-80.00,-60.00)}] rectangle +(20.00,20.00);
\path[line width=0.02mm, draw=L, fill=F] (-60.00,-40.00) [rotate around={270:(-60.00,-40.00)}] rectangle +(20.00,20.00);
\path[line width=0.02mm, draw=L, fill=F] (-20.00,0.00) [rotate around={270:(-20.00,0.00)}] rectangle +(20.00,20.00);
\path[line width=0.02mm, draw=L, fill=F] (-40.00,-20.00) [rotate around={270:(-40.00,-20.00)}] rectangle +(20.00,20.00);
\definecolor{F}{rgb}{0,0,0}
\path[line width=0.05mm, draw=L, fill=F] (-45.00,-45.00) circle (0.09mm);
\path[line width=0.05mm, draw=L, fill=F] (-50.00,-50.00) circle (0.09mm);
\path[line width=0.05mm, draw=L, fill=F] (-55.00,-55.00) circle (0.09mm);
\definecolor{F}{rgb}{0.827,0.827,0.827}
\path[line width=0.02mm, draw=L, fill=F] (-80.00,60.00) [rotate around={270:(-80.00,60.00)}] rectangle +(20.00,20.00);
\draw(-70.00,48.00) node[anchor=base]{\fontsize{6}{7}\selectfont $1$};
\path[line width=0.02mm, draw=L, fill=F] (60.00,-40.00) [rotate around={270:(60.00,-40.00)}] rectangle +(20.00,20.00);
\path[line width=0.02mm, draw=L, fill=F] (20.00,0.00) [rotate around={270:(20.00,0.00)}] rectangle +(20.00,20.00);
\path[line width=0.02mm, draw=L, fill=F] (40.00,-20.00) [rotate around={270:(40.00,-20.00)}] rectangle +(20.00,20.00);
\definecolor{F}{rgb}{0,0,0}
\path[line width=0.05mm, draw=L, fill=F] (44.00,-24.00) circle (0.09mm);
\path[line width=0.05mm, draw=L, fill=F] (50.00,-30.00) circle (0.09mm);
\path[line width=0.05mm, draw=L, fill=F] (55.00,-35.00) circle (0.09mm);
\draw(30.00,-12.00) node[anchor=base]{\fontsize{6}{7}\selectfont $1$};
\draw(70.00,-52.00) node[anchor=base]{\fontsize{6}{7}\selectfont $2k$};
\draw(-10.00,-12.00) node[anchor=base]{\fontsize{6}{7}\selectfont $1$};
\draw(-30.00,-32.00) node[anchor=base]{\fontsize{6}{7}\selectfont $2$};
\draw(-70.00,-72.00) node[anchor=base]{\fontsize{6}{7}\selectfont $2k$+$1$};
\end{tikzpicture}%
}
\caption{Block structure of the normal form of $\vec{H}|_{E_\lambda}$ for a pair of Jordan blocks of even order (case a) and odd order (case b). In case (a), $\dim E_{\lambda} = 4k$, and each box denotes the presence of a non-vanishing $2\times 2$ block. In case (b), $\dim E_{\lambda} = 4k+2$, and each box denotes the presence of a non-vanishing $1\times 1$ block. All other entries are zero.}
\end{figure}
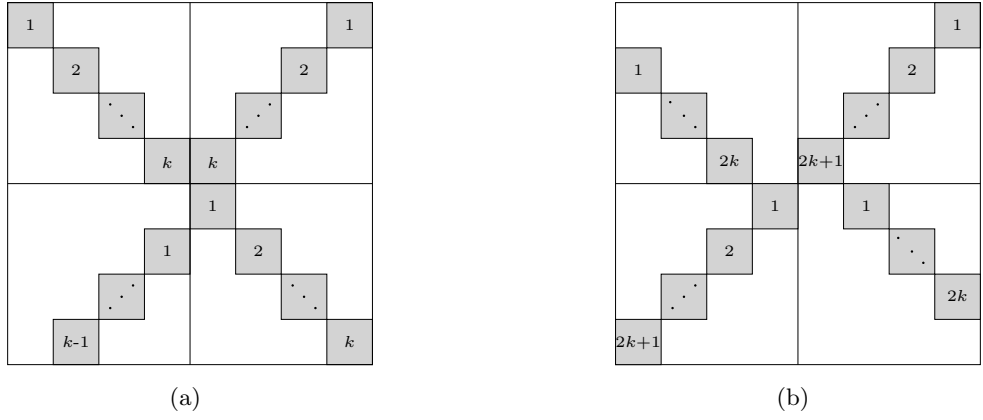
Notice that, for what follows, we do not need to know the explicit form of each box. If $k$ is even, we choose $\Gamma=\{(p,x)\in \mathbb{R}^{4k}|\, p_{k+1}=\ldots=p_{2k}=x_1=\ldots=x_{k}=0\}$ and if $k$ is odd we set $\Gamma=\{(p,x)\in \mathbb{R}^{4k}|\, p_{k+2}=\ldots=p_{2k}=x_1=\ldots=x_{k+1}=0\}$. It is a simple check that, in both cases, $\Gamma$ is a $2k$-dimensional $\vec{H}$-invariant space, which is also isotropic by construction, and thus Lagrangian (since $\dim E_\lambda = 4k$).

Now, suppose that the Jordan block corresponding to $\lambda$ has odd order $2k+1$. Thus $H|_{E_\lambda}$ can be written as in Eq.~\eqref{ham_normal_form_2k+1} and $\vec{H}\rvert_{E_\lambda} = -\J\H|_{E_\lambda}$ has the structure, in coordinates $(p,x) \in \R^{4k+2}$, displayed in Fig.~\ref{f:b}. Once again, we stress that we do not need the explicit form of each box. By choosing $\Gamma=\{(p,x)\in \mathbb{R}^{2n}|\, p_1=\ldots=p_{k+1}=x_{k+1}=\ldots=x_{2k+1}=0\}$, we get the required subspace.
\end{proof}

\begin{prop}\label{prop_oddjordan}
Let $\vec{H} \in \HH$. Suppose there exists at least one Jordan block of odd order corresponding to a pure imaginary eigenvalue of $\vec{H}$. Thus the Jacobi curve has infinitely many conjugate times.
\end{prop}

\begin{proof}
We will reduce the problem to the diagonalizable case by studying the curve in a reduced space $\Sigma^\Gamma=\Gamma^\angle/\Gamma$. Let $\pm\lambda_1,\dots, \pm\lambda_m$ be the pure imaginary eigenvalues of $\vec{H}$ and let us consider, for each $i$, the quotient spaces $E_{\lambda_i}^{\Gamma_i}:=E_{\lambda_i}\cap\Gamma_i^\angle / \Gamma_i$, where the subspaces $\Gamma_i\subset E_{\lambda_i}$ are as in Lemma \ref{lem_Jordan_blocks_isotropic_subspace}. Notice that $\dim E_{\lambda_i}^{\Gamma_i}=0$ or $2$ depending on whether the Jordan block corresponding to $\lambda_i$ is even or odd, respectively. Now set $\Gamma=\Gamma_1\oplus\dots\oplus\Gamma_m\oplus\Gamma_+ $, where $\Gamma_+$ as in Lemma~\ref{lem_lagrangian_invariant}. Hence $\Sigma^\Gamma=E_{\lambda_1}^{\Gamma_1}\oplus\dots \oplus E_{\lambda_m}^{\Gamma_m}$, so if there is at least one $\lambda_i$ for which the corresponding Jordan block has odd order then $\vec{H}\rvert_{\Sigma^\Gamma}$ has nonempty pure imaginary spectrum and it is diagonalizable. Moreover, since the original Jacobi curve is ample and monotone, the reduced Jacobi curve $J^\Gamma(\cdot)$ is ample and monotone too by Lemma~\ref{l:ampleproj}. Thus the result follows from Proposition \ref{prop_pure_imaginary_spectrum}.
\end{proof} 

\begin{prop}\label{prop_noconjtimes}
Let $\vec{H} \in \HH$. If all Jordan blocks of $\vec{H}$ corresponding to pure imaginary eigenvalues are of even order, the Jacobi curve has no conjugate times.
\end{prop}
\begin{rem}
This proposition applies, in particular, when there are no pure imaginary eigenvalues.
\end{rem}

\begin{proof}
By Lemma~\ref{lem_lagrangian_invariant_no_conj_points} it is enough to find an $\vec{H}$-invariant Lagrangian subspace $\Gamma\subset\Sigma$. Let $\pm\lambda_1,\dots, \pm\lambda_m$ be the pure imaginary eigenvalues of $\vec{H}$. By Lemma~\ref{lem_Jordan_blocks_isotropic_subspace} there exists a Lagrangian $\vec{H}$-invariant subspace $\Gamma_i\subset E_{\lambda_i}$ for each $i$. Set $\Gamma=\Gamma_1\oplus\dots\oplus\Gamma_m\oplus\Gamma_+ $, where $\Gamma_+$ is as in Lemma~\ref{lem_lagrangian_invariant}.
\end{proof}

\paragraph*{Acknowledgements.}
The first author has been supported by the grant of the Russian Federation for the state support of research, Agreement No 14 B25 31 0029. The second author has been supported by the European Research Council, ERC StG 2009 “GeCoMethods”, contract number 239748, by INdAM (GDRE CONEDP) and by the Institut Henri Poincar\'e, Paris, where part of this research has been carried out.

{\footnotesize
\bibliographystyle{abbrv}
\bibliography{../biblio/LQ-bib}
}

\end{document}